\documentclass[reqno, 12pt]{amsart}

\makeatletter
\let\origsection=\section

\def\section{\@ifstar{\origsection*}{\mysection}}

\def\mysection{\@startsection{section}{1}\z@{.7\linespacing\@plus\linespacing}{.5\linespacing}{\normalfont\scshape\centering\S}}
\makeatother

\usepackage{amsmath,amssymb,amsthm}
\usepackage{mathrsfs}
\usepackage{dsfont}
\usepackage{mathabx}\changenotsign
\usepackage{enumerate}
\usepackage{microtype}
\usepackage{color}

\usepackage{xcolor}
\usepackage[backref=page]{hyperref}
\hypersetup{
    colorlinks,
    linkcolor={red!60!black},
    citecolor={green!60!black},
    urlcolor={blue!60!black}
}

\usepackage{tikz}
\usepackage{subcaption}
\usetikzlibrary{graphs}

\usepackage{bookmark}

\usepackage[abbrev,msc-links,backrefs]{amsrefs}

\usepackage{doi}

\renewcommand{\PrintDOI}[1]{\doi{#1}}

\usepackage[T1]{fontenc}
\usepackage{lmodern}

\usepackage[utf8]{inputenc}
\usepackage[english]{babel}

\usepackage{geometry}
\usepackage{comment}
\usepackage{enumitem}

\def\eps{\varepsilon}

\def\KE{\mathop{\text{\rm KE}}\nolimits}
\def\int{\mathop{\text{\rm int}}\nolimits}
\def\bip{\mathop{\text{\rm bip}}\nolimits}

\let\polishlcross=\l
\def\l{\ifmmode\ell\else\polishlcross\fi}

\let\emptyset=\varnothing
\let\setminus=\smallsetminus

\makeatletter
\def\moverlay{\mathpalette\mov@rlay}
\def\mov@rlay#1#2{\leavevmode\vtop{%
   \baselineskip\z@skip \lineskiplimit-\maxdimen
   \ialign{\hfil$\m@th#1##$\hfil\cr#2\crcr}}}
\newcommand{\charfusion}[3][\mathord]{
    #1{\ifx#1\mathop\vphantom{#2}\fi
        \mathpalette\mov@rlay{#2\cr#3}
      }
    \ifx#1\mathop\expandafter\displaylimits\fi}
\makeatother

\newtheorem{theorem}{Theorem}
\newtheorem{lemma}[theorem]{Lemma}

\newtheorem{definition}[theorem]{Definition}
\newtheorem{fact}[theorem]{Fact}

\begin{document}

\title[]{Three-colour bipartite Ramsey number for\\ graphs with small bandwidth}

\author[]{G.~O.~Mota}
\address{Centro de Matem\'atica, Computa\c c\~ao e Cogni\c c\~ao, Universidade Federal do ABC, Santo Andr\'e, Brazil}
\email{g.mota@ufabc.edu.br}

%\thanks{G.~O.~Mota was supported by FAPESP (2009/06294-0 and 2012/00036-2).}

\begin{abstract}
We estimate the $3$-colour bipartite Ramsey number for balanced bipartite graphs~$H$ with small bandwidth and bounded maximum degree.
More precisely, we show that the minimum value of $N$ such that in any $3$-edge colouring of $K_{N,N}$ there is a monochromatic copy of $H$ is at most $\big(3/2+o(1)\big)|V(H)|$.
In particular, we determine asymptotically the $3$-colour bipartite Ramsey number for balanced grid graphs.\\
\end{abstract}

\keywords{Bipartite Ramsey, Bandwidth, Regularity Lemma}

\maketitle 

\section{Introduction}

Given graphs $H_1, \ldots , H_r$, the \emph{Ramsey number} $R(H_1, \ldots, H_r)$ is the smallest integer $N$ such that  any complete graph $K_N$  with edges coloured with colours $1,\ldots,r$ contains a copy of some $H_i$ ($1\leq i\leq r$) where all edges of $H_i$ have colour $i$.
The existence of Ramsey numbers follows from Ramsey's Theorem~\cite{Ra30}.

There are many results about Ramsey numbers $R(H_1, \ldots, H_r)$ for particular families of graphs $\{H_1, \ldots, H_r\}$ (see, e.g.,~\cite{CoFoSu15,GrRoSp90,Ra94}).
However, determining Ramsey numbers (even asymptotically) seems to be a difficult problem.
For two colours, it was proved by Gerencs\'er and Gy\'arf\'as~\cite{GeGy67} that $R(P_n, P_n) = \left\lfloor (3n-2)/2 \right\rfloor$, where $P_n$ is the path with $n\geq 2$ vertices.
The Ramsey number $R(T,T)$ for general trees $T$ with some restriction on the degree was asymptotically determined in~\cite{HaLuTi02}.
For cycles $C_n$ the Ramsey number $R(C_n,C_n)$ was determined, independently, in~\cite{FaSc74} and~\cite{Ro73}.

For three colours the situation is more complicated.
Figaj and {\L}uczak~\cite{FiLu07} determined $R(P_n,P_n,P_n)$ asymptotically and 
Gyárfás, Ruszinkó, Sárk\"ozy and Szemerédi~\cite{GyRuSaSz07b} strengthen this result determining the $3$-colour Ramsey number for paths exactly. 
The $3$-colour Ramsey number for even cycles was first determined asymptotically by Figaj and {\L}uczak~\cite{FiLu07}.
Later, it was determined exactly by Be\-ne\-vi\-des and Skokan~\cite{BeSk09} for sufficiently large even cycles.
{\L}uczak~\cite{Lu99} determined the $3$-colour Ramsey number for odd cycles asymptotically and Kohayakawa, Simonovits and Skokan~\cite{KoSiSk05} found the exact value of it for long odd cycles.

%For results about $k$-colour Ramsey numbers of paths and cycles, we refer the reader to~\cite{Yongqi, Yuansheng, Feng and Bingxi,     Davies--Jenssen--Roberts,		Jenssen--skokan,        Day--Johnson}.

We say a graph $H=(V_H,E_H)$ has {\em bandwidth} at most $b$ it there is a labelling $v_1,\ldots, v_n$ of the vertices of $V_H$ such that $|i-j|\leq b$ for every edge $v_iv_j\in E_H$.
The $3$-colour Ramsey number for balanced bipartite graphs $H$ with small bandwidth and bounded maximum degree (see Definition~\ref{def:betadelta} below) was estimated in~\cite{MoSaScTa15}.
Such estimation determines asymptotically the $3$-colour Ramsey number for grid graphs.

\begin{definition}\label{def:betadelta}
A bipartite graph $H=(V_H,E_H)$ is a \emph{$(\beta,\Delta)$-graph} if it has bandwidth at most $\beta |V_H|$ and $\Delta(H)\leq \Delta$. 
Furthermore, $H$ is a \emph{balanced $(\beta,\Delta)$-graph} if there is a proper $2$-colouring $\chi\colon V_H
\to [2]$ such that ${\big||\chi^{-1}(1)|-|\chi^{-1}(2)|\big| \leq \beta |\chi^{-1}(2)|}$.
\end{definition}

We remark that, for any fixed $\beta>0$, sufficiently large planar graphs with maximum degree at most 
$\Delta$ are $(\beta, \Delta)$-graphs (see~\cite{BoPrTaWu10}) and it is easy to show that sufficiently large grid graphs with $n$ vertices are $(\beta, 4)$-graphs (see, e.g.,~\cite{MoSaScTa15}).
We are interested in estimating, for balanced bipartite graphs with small bandwidth and bounded degree, the following generalization of Ramsey numbers that concerns edge-colourings of bipartite complete graphs:
given graphs $H_1, \ldots , H_r$, the \emph{bipartite Ramsey number} $R^{\bip}(H_1, \ldots, H_r)$ is the smallest integer $N$ such that  any $K_{N,N}$  with edges coloured with colours $1,\ldots,r$ contains a copy of some $H_i$ ($1\leq i\leq r$) where all edges of $H_i$ have colour~$i$.

For two colours, independently, Gyarfás and Lehel~\cite{GyLe73} and Faudree and Schelp~\cite{FaSc75} determined the exact bipartite Ramsey number for paths, showing that $R^{\bip}(P_n,P_n)=n$ for odd $n$, and  $R^{\bip}(P_n,P_n)=n-1$ for even $n$.
In~\cite{ZhSu11}, Zhang and Sun proved that $R^{\bip}(C_{2n},C_4)=n+1$, and in~\cite{ZhSuWu13} Zhang, Sun and Wu proved that $R^{\bip}(C_{2n},C_6)=n+2$ for $n\geq 4$.
The methods developed by Buci\'c, Letzter and Sudakov in~\cite{BuLeSu18+} proves that~$R^{\bip}(C_{2n},C_{2m})=\big(1+o(1)\big)(n+m)$ for all positive $n$ and $m$, determining asymptotically the bipartite Ramsey number of every pair of cycles.
For complete bipartite graphs, the best known bounds for $R^{\bip}(K_{n,n},K_{n,n})$ differ exponentially~(see~\cite{Co08,HaHe98}).
More estimates for bipartite Ramsey numbers can be seen in~\cite{BeSc76,LiTaZa05,Th82}.

Buci\'c, Letzter and Sudakov proved in~\cite{BuLeSu18+} that $R^{\bip}(C_{n},C_{n},C_{n})=\big(3/2 + o(1)\big)n$ for even $n$, which implies that $R^{\bip}(P_{n},P_n,P_n)=\big(3/2 + o(1)\big)n$.
To see that this result is asymptotically best possible, consider the following $3$-edge colouring of $K_{N,N}$ with $N=3(n/2 - 1)$: divide one of the classes of $K_{N,N}$ into three parts, $V_1$, $V_2$, and $V_3$, each of size $n/2-1$, and give colour $i$ to every edge incident to vertices in $V_i$, for $i\in[3]$.
Let $n$ be even and let $H$ be an $n$-vertex bipartite graph with $n/2$ vertices in each class of the bipartition (this includes $P_n$ and $C_n$).
Clearly, the given edge colouring of $K_{N,N}$ contains no monochromatic copy of $H$.
This implies that 
\begin{equation}\label{lower}
R^{\bip}(H,H,H)\geq \frac{3n}{2}-2.
\end{equation}
In particular,~\eqref{lower} implies that $R^{\bip}(C_{n},C_{n},C_{n})$ and $R^{\bip}(P_{n},P_n,P_n)$ are at least $3n/2-2$ for even~$n$.

Using a result proved in~\cite{BuLeSu18+} together with the methods applied in~\cite{MoSaScTa15}, we prove that the $3$-colour bipartite Ramsey number of $n$-vertex balanced bipartite graphs with small bandwidth and bounded maximum degree is at most $(3/2+o(1))n$.
The next theorem is our main result.

\begin{theorem}\label{thm:threecolour}
For every $\gamma>0$ and every natural number $\Delta$, there exist a
constant $\beta>0$ and natural number $n_0$ such that for every
balanced $(\beta,\Delta)$-graph $H$ on $n\geq n_0$ vertices we
have
$$
R^{\bip}(H,H,H) \leq (3/2 + \gamma) n.
$$
\end{theorem}

%To see that the bound in Theorem~\ref{thm:threecolour} is asymptotically best possible, note that since any $n$-vertex bipartite $H$ with $n$ vertices has one class of its bipartition with at least $n/2$ vertices, the following 3-edge colouring of $K_{N,N}$ with $N\leq 3n/2$ 

We remark that Theorem~\ref{thm:threecolour} applies to grid graphs and bipartite planar graphs of bounded degree, where a grid graph $G_{a,b}$ is the graph with vertex set ${V=[a]\times[b]}$ such that there is an edge between two vertices if they are equal in one coordinate and consecutive in the other.
In particular, together with the lower bound~\eqref{lower}, Theorem~\ref{thm:threecolour} determines asymptotically the bipartite Ramsey number of balanced grid graphs, i.e., $R^{\bip}(G_{a,a},G_{a,a},G_{a,a}) = (3/2 + o(1)) a^2$.

In Section~\ref{sec:aux} we present some auxiliary results necessary to prove Theorem~\ref{thm:threecolour}, which is proved in details in Section~\ref{sec:main}.

\section{Auxiliary results}\label{sec:aux}

In this section we present some results that will be used to prove Theorem~\ref{thm:threecolour}.
First, in Section~\ref{subsec:balBlocks}, we state a result (see Lemma~\ref{lemma:balancedVectors} below) that will allow us to order any partition $\mathcal{W}=W_1,\ldots,W_{\hat\ell}$ (with classes having almost the same size) of a balanced bipartite graph $H$ in a way that the subgraph induced by any large contiguous set of classes in this order is also balanced.
In Section~\ref{subsec:reg} we discuss the regularity method, where we state the version of the regularity and embedding lemmas we need in our proof.
Finally, we prove a result (Lemma~\ref{lemma:FindTheNiceTree}) ensuring that, for sufficiently large $N$, any $3$-edge colouring of $K_{N,N}$ contains a large monochromatic subgraph that is some regular blow-up of a tree containing a matching (see Definition~\ref{def:shape} below).

\subsection{Local balancedness}\label{subsec:balBlocks}

Let $H=(V_H,E_H)$ be a graph with $V_H=\{v_1,\ldots,v_n\}$ and consider a $2$-colouring ${\chi_H\colon V_H\to[2]}$.
Given a subset of vertices $W\subset V_H$, denote by $C_1$ and $C_2$, respectively, the number of vertices with colour $1$ and $2$ under $\chi_H$.
We say $\chi_H$ is a $\beta$\emph{-balanced} colouring of $V_H$~if 
$$
1-\beta\leq \frac{C_1(W)}{C_2(W)}\leq 1+\beta.
$$
Given positive integers $\hat\ell$ and $1\leq a<b\leq \hat\ell$, consider a partition $\mathcal{W}=\{W_1,\ldots, W_{\hat\ell}\}$ of $V_H$, and a permutation $\sigma\colon[\hat\ell]\to[\hat\ell]$, we define $C_i(\mathcal{W},\sigma,a,b)=\sum_{j=a}^{b} C_i(W_{\sigma(j)})$ for $i=1,2$, i.e., $C_i(\mathcal{W},\sigma,a,b)$ is the number of vertices in $W_{\sigma(a)}\cup W_{\sigma(a+1)}\cup\dots\cup W_{\sigma(b)}$ with colour $i$.
If it is clear what partition $\mathcal{W}$ we are considering then we write simply $C_i(\sigma,a,b)$.

Given a graph $H=(V_H,E_H)$, let $\chi_H\colon V_H\to[2]$ be a colouring of $V_H$ such that $H$ is globally 
balanced. Roughly speaking, the next lemma states that, given a partition $\mathcal{W}$ of~$V_H$ into parts of almost the same size, the parts of $\mathcal{W}$ can be ordered in a way that, after such ordering, in every (not so small) contiguous set of parts, the difference between the number of vertices $w$ with $\chi_H(w)=1$ and those $w$ with $\chi_H(w)=2$ is small.

\begin{lemma}[{\cite[Lemma 2.11]{MoSaScTa15}}]\label{lemma:balancedVectors}
For every $\xi>0$ and every integer $\hat \ell\geq 1$ there exists $n_0$ such that if $H=(V_H,E_H)$ is a graph with 
$n\geq n_0$, then for every \text{$\beta$-balanced} $2$-colouring $\chi_H$ of $V_H$ with
$\beta\leq 2/\hat\ell$, and every partition of $V_H$ into parts $W_1,\ldots,W_{\hat\ell}$ with $|W_1|\leq\ldots 
\leq|W_{\hat\ell}|\leq |W_1|+1$ there 
exists a permutation
$\sigma\colon[\hat\ell]\to[\hat\ell]$ such that for every pair of integers $1\leq a<b\leq \hat\ell$ with $b-a\geq 
7/\xi$, we have
\begin{equation*}
|C_1(\sigma,a,b)-{C_2(\sigma,a,b)}|\leq \xi {C_2(\sigma,a,b)},
\end{equation*}
\end{lemma}

\subsection{Regularity and embedding lemmas}\label{subsec:reg}

Given an $n$-vertex graph $G=(V_G,E_G)$ and non-empty disjoint subsets $A$, $B\subset V_G$, we denote by
$e_G(A,B)$ the number of edges of $G$ with one endpoint in $A$ and
the other in $B$, and $d_G(A,B)=e_G(A,B)/(|A||B|)$ is the \emph{density} of $G$ between $A$ and $B$. 
The pair $(A,B)$ is called $\varepsilon$\emph{-regular} if
for all $X\subset A$ and $Y\subset B$ with $|X|\geq\varepsilon|A|$ and $|Y|\geq\varepsilon|B|$ we have
\begin{equation*}
|d_G(X,Y)-d_G(A,B)|<\varepsilon. 
\end{equation*}

The proof of Theorem~\ref{thm:threecolour} is based on the following $3$-colour version
of the Regularity Lemma for bipartite graphs (see~\cite{Sz75} for the original version of the Regularity Lemma).

\begin{lemma}[Regularity Lemma]\label{lemma:Regularity}
For every $\varepsilon>0$ and every integer $k_0>0$ there exists
a positive integer $K_0(\varepsilon, k_0)$ such that for $n\geq K_0$ the
following holds. For all bipartite graphs $G_1$, $G_2$ and $G_3$ with
the same bipartition $(L,R)$ such that $|L|=|R|=n$, there is a partition of $L\cup R$
into $2k+1$ classes $V_0,V_1,V_2,\dots,V_{2k}$ such that
\begin{itemize}
\item [(i)] $k_0\leq 2k\leq K_0$,
\item [(ii)] $|V_1|=|V_2|=\dots=|V_{2k}|$,
\item [(iii)] $|V_0\cap L|=|V_0\cap R|<\varepsilon n$,
\item [(iv)] every $V_i$ is contained in either $L$ or $R$, for every $1\leq i\leq 2k$,
\item [(v)] for every $V_i$ with $1\leq i\leq k$, apart from at most $\varepsilon (2k)$ classes $V_j$, the pair $(V_i,V_j)$ is $\varepsilon$-regular in $G_1$, $G_2$ and $G_3$.
\end{itemize}
\end{lemma}

Before stating the embedding results that we need, let us give a few more definitions.
A bipartite graph $G=(A,B;E_G)$ is $(\varepsilon,d)$\emph{-regular} if it is
$\varepsilon$-regular and $d_G(A,B)\geq d$. 
Furthermore, $G$ is called $(\varepsilon,d)$\emph{-super-regular} if it is $\eps$-regular and ${\deg_G(a)>d|B|}$ for 
all $a\in A$ and ${\deg_G(b)>d|A|}$ for all $b\in B$.
For a graph $G=(V_G,E_G)$, a partition $\{V_1,\ldots,V_s\}$ of $V_G$ is $(\varepsilon,d)$\emph{-regular} (resp. \emph{super-regular}) \emph{on a graph} $R=(V_R,E_R)$ with vertex set $[s]$ if, for every edge $ij\in E_R$, the
bipartite subgraph of $G$ induced by the pair $\{V_i,V_j\}$ is $(\varepsilon,d)$-regular (resp. super-regular).
The graph $R$ is the \emph{reduced graph} of the partition $\{V_1,\ldots,V_s\}$ (or of the graph $G$). 
We refer the reader to~\cite{KoSi96,KoShSiSz02} for surveys devoted to the Regularity Lemma and its applications.

We will make use of the following two simple facts that can be easily proved using the definitions of regular and super-regular pairs.

\begin{fact}\label{fact:Slicing}
Let $G=(A,B;E_G)$ be an $\varepsilon$-regular bipartite graph and let $A'\subset A$ and ${B'\subset B}$ with 
$|A'|\geq \alpha |A|$ and 
$|B'|\geq \alpha |B|$ for some $\alpha>\varepsilon$. Then the graph ${G'=\big(A',B';E_G(A',B')\big)}$ is 
$\varepsilon'$-regular such that
${|d_G(A,B)-d_{G'}(A',B')|<\varepsilon}$, where $\varepsilon'=\max\{\varepsilon/\alpha,2\varepsilon\}$.
\end{fact}

Given a matching $E_M$ of a graph $G=(V_G,E_G)$, we say that a vertex $v\in V_G$ is a \emph{vertex of $E_M$} if there is an edge of $E_M$ that is incident to $v$.

\begin{fact}\label{fact:SuperSlicing}
Consider a graph $G=(V_G,E_G)$ with an $(\varepsilon,d)$-regular partition $\{V_1,\ldots,V_s\}$ of $V$ with $|V_i|=m$ for $1\leq i\leq s$.
Let $T$ be a graph on vertex set $[s]$ contained in the corresponding $(\varepsilon,d)$-reduced graph of $\{V_1,\ldots,V_s\}$ and let $E_M$ be a matching contained in~$T$. Then for each vertex $i$ of $E_M$, the associated set $V_i$ in $G$ contains a subset $V_i'$ of size $(1-\varepsilon r)m$ such that for every edge $ij$ of $E_M$ the bipartite graph $(V_i',V_j';E_G(V_i',V_j'))$ is $(\varepsilon/(1-\varepsilon r),d-(1+r)\varepsilon)$-super-regular.
\end{fact}

In the proof of Theorem~\ref{thm:threecolour} we need to apply the so-called Blow-up Lemma~\cite{KoSaSz97} (see also \cite{KoSaSz98, RoRu99,RoRuTa99}) to embed a $(\beta,\Delta)$-graph $H$ into a monochromatic subgraph $G$ of $K_{N,N}$.
To be able to do such embedding, we need to show that $G$ and $H$ have ``compatible'' partitions, which we define below.

\begin{definition}\label{def:compatiblePartitions}
Let $H = (V_H, E_H)$ be a graph. Let $T = ([s], E_T)$ be a tree and $M = ([s], E_{M})$ be a subgraph of $T$ where $E_M$ is a matching.
Given a partition $\{W_1,\ldots,W_s\}$ of $V_H$, let $U_i$, for $i\in[s]$, be the set of vertices in $W_i$ with neighbours in some $W_j$ with $ij \in E_{T}\setminus E_{M}$.
Set $U= \bigcup U_i$ and $U'_i= N_H(U)\cap(W_i\setminus U)$.

We say that $\{W_1,\ldots,W_s\}$ is $(\varepsilon,T,M)$-compatible with a vertex partition $\{V_1,\ldots,V_s\}$ of
a graph $G=(V_G,E_G)$ if the
following four conditions hold.
\begin{itemize}
 \item[(i)]If $xy\in E_H$ with $x\in W_i$ and $y\in W_j$, then $ij\in E_T$ for all $i, j \in [s]$,
 \item[(ii)]$|W_i|\leq |V_i|$ for all $i\in [s]$,
 \item[(iii)]$|U_i|\leq \varepsilon |V_i|$ for all $i\in [s]$,
 \item[(iv)]$|U'_i|,|U'_j|\leq \varepsilon \min\{|V_i|,|V_j|\}$ for all $ij\in E_M$.
\end{itemize}
\end{definition}

%We remark that for connected graphs $H$ and for every vertex $i$ of $T$ which is not covered by $E_M$ we have $U_i=W_i$ and $U'_i=\emptyset$.

Considering the setup of Definition~\ref{def:compatiblePartitions}, roughly speaking, the following corollary of the Blow-up Lemma states that bounded degree graphs $H$ can be embedded into $G$, whenever $H$ and $G$ admit compatible partitions and the partition of $G$ is sufficiently dense and regular on $T$ and super-regular on $M$.

\begin{lemma}[Embedding Lemma \cite{Bo09,BoHeTa10}]\label{lemma:GeneralEmbedding}
For all positive $d, \Delta$ there is a positive constant $\varepsilon$ such that the following holds. Let $G = (V_G, E_G)$ be an $N$-vertex
graph that has a partition $\{V_1,\ldots,V_s\}$ of $V_G$ with $(\varepsilon, d)$-reduced graph $T$ on~$[s]$ which is
$(\varepsilon, d)$-super-regular on a graph $M \subset T$. Further, let $H = (V_H, E_H)$ be an $n$-vertex graph with 
maximum degree at most $\Delta$ and $n\leq N$ that has a vertex partition $\{W_1,\ldots,W_s\}$ of $V_H$ which is $(\varepsilon,T,M)$-compatible with $\{V_1,\ldots,V_s\}$. Then $H\subset G$.
\end{lemma}

\subsection{Connected matchings and regular blow-ups}\label{subsec:blowtree}

The main result of this section (see Lemma~\ref{lemma:FindTheNiceTree} below), a key lemma in the proof of Theorem~\ref{thm:threecolour}, shows that in any edge colouring of $K_{N,N}$ there exists a dense monochromatic subgraph $G_1$ of $K_{N,N}$ that is regular on a ``large''  tree $T$ that has some structural properties that allow us to embed $H$ into $G_1$.

A {\em connected matching} in a graph $R$ is a matching $E_M$ with all its edges in the same connected component of $R$.
A powerful technique introduced by {\L}uczak in~\cite{Lu99} reduces some Ramsey problems for cycles and paths to problems about connected matchings (see, e.g., \cite{BuLeSu18+,FiLu07,GyRuSaSz07b,GyRuSaSz07,HaLuPeRoRuSiSk06}).
Here we follow the strategy used in~\cite{MoSaScTa15} that applies the connected matching technique to graphs more general than cycles and paths.
Lemma~\ref{lemma:Matching} below, recently proved by Buci\'c, Letzter and Sudakov~\cite{BuLeSu18+}, is essential in the proof of Lemma~\ref{lemma:FindTheNiceTree}, the main result of this section.
Lemma~\ref{lemma:Matching} will allow us to find a ``large'' connected matching in some $3$-coloured almost complete bipartite reduced graph.

% falar que tem os alfa e beta no lema original
\begin{lemma}\label{lemma:Matching}
For every positive $\eps< 1/(3\cdot10^5)$ there exists a natural number $k_0$ such that the following holds for every $k'\geq k_0$.
Let $R=(A,B;E_R)$ be a bipartite graph with $|A|=|B|=k$, where $k\geq (3+(3\cdot10^5)\eps)k'$, and suppose that every vertex in $A$ has at most $\eps k'$ non-neighbours in $B$ and vice versa.
Then, in every $3$-colouring of $E_R$ there exists a monochromatic connected matching with at least $k'$ edges.
\end{lemma}

We remark that the bound $1/(3\cdot10^5)$ in Lemma~\ref{lemma:Matching} is weaker then the bound proved in~\cite{BuLeSu18+}, but we stated it in this form for simplicity.
The following definition plays an important role in our proof.

\begin{definition}\label{def:shape}
Let $G=(V_G,E_G)$ be a graph and consider an edge-colouring $\chi_G$ of $E_G$.
Given positive $\eps$ and $d$, and positive integers $\ell$, $\ell'$ and $k$, we say that $G$ has an \emph{$(\eps,d)$-regular $(\ell,\ell',k)$-cm-shape under $\chi_G$} if there exists a tree $T$ on vertex set $\{x_1,\dots,x_\ell,y_1,\dots,y_\ell,z_1,\dots,z_{\ell'}\}$ containing a matching $E_M=\{x_iy_i\colon i=1,\dots,\ell\}$ with an 
even distance in $T$ between any $x_i$ and $x_j$ for all $1\leq i<j\leq \ell$, and there exists a partition $\{V_0,V_1,\ldots,V_k\}$ of $V_G$ with $|V_1|=\ldots=|V_k|\geq (1-\eps)|V_G|/k$ such that the spanning subgraph of $G$ spanned by one of the colours is $(\varepsilon,d)$-regular on $T$.
\end{definition}

The next lemma is the main result of this section. 
Given a colouring $\chi_{K_{N,N}}\colon E(K_{N,N})\to[3]$, we denote by $G_i$ the spanning subgraph of $K_{N,N}$ containing only edges with colour $i$, for $i\in [3]$.

\begin{lemma}\label{lemma:FindTheNiceTree}
For every $\varepsilon <1/(24\cdot10^5)$ there exists $K_0$ such that for
all $N \geq K_0$ and for every colouring $\chi_{K_{N,N}}\colon E(K_{N,N})\to[3]$, there exist integers $\ell,\ell',k$ with 
$\ell,\ell'\leq k \leq K_0$ and $\ell\geq k/\big(3+(24\cdot10^5 \eps)\big)$ such that $K_{N,N}$ has an $(\eps,1/3)$-regular $(\ell,\ell',2k)$-cm-shape under $\chi_{K_{N,N}}$.
\end{lemma}

\begin{proof}
Fix $\eps<1/(24\cdot10^5)$.
From Lemma~\ref{lemma:Matching} applied with $8\eps$, we obtain $k_0$.
Now let $K_0$ be obtained by an application of the Regularity Lemma (Lemma \ref{lemma:Regularity}) with parameters $\eps$ and $4 k_0$. 
Finally let ${N\geq K_0}$ be given and consider an arbitrary $3$-colouring $\chi_{K_{N,N}}\colon E(K_{N,N})\to[3]$ of $E(K_{N,N})$.

From the Regularity Lemma, we know that there is a partition $\{V_0,V_1,\ldots,V_{2k}\}$ of the vertices of $K_{N,N}$ with $k_0\leq 2k\leq K_0$ such that the following properties hold: $V_0$ intersects both parts of the bipartition of $K_{N,N}$ in the same number of vertices, $|V_1|=\ldots|V_{2k}|=m\geq (1-\varepsilon)N/k$, every $V_i$ ($1\leq i\leq 2k$) is completely contained in one of the parts of the bipartition of $K_{N,N}$, and for every $V_i$ ($1\leq i\leq 2k$), apart from at most $2\varepsilon k$ classes $V_j$, the pair $(V_i,V_j)$ is $\varepsilon$-regular in $G_1$, $G_2$ and $G_3$.

Let $R=(A,B;E_R)$ be the reduced graph with vertex set $[2k]$ such that $ij\in E_R$ if and only if 
$\{V_i,V_j\}$ is $\varepsilon$-regular in each of $G_1$, $G_2$ and $G_3$.
Since every $V_i$ ($1\leq i\leq k$) is completely contained in one of the parts of the bipartition of $K_{N,N}$ and the intersection of each one of these parts with $V_0$ has the same cardinality, $|A|=|B|=k$.

Let $\chi_R\colon E(R)\to[3]$ be a colouring of the edges of $R$ such that $\chi_R(i,j)=s$ if $s\in[3]$ is the 
biggest integer with $|E_{G_s}(V_i,V_j)| \geq |E_{G_r}(V_i,V_j)|$ for $1\leq r\leq 3$,
i.e., the edge $ij$ receives one of the majority colours (considering $\chi_{K_{N,N}}$) in $E_{K_{N,N}}(V_i,V_j)$.

Put $k' = k/(3+(24\cdot 10^5\eps))$.
Note that since $k\geq 4k_0$, we have $k'\geq k_0$.
This together with the facts that $k\geq \big(3+(3\cdot 10^5(8\eps))\big)k'$ and every vertex in $A$ (resp. $B$) has at most $2\eps k\leq (8\eps) k'$ non-neighbours in $B$ (resp. $A$), allow us to use Lemma~\ref{lemma:Matching}, which guarantees the existence of a monochromatic tree $T=(V_T,E_T)$ with $|E_T|=\ell\geq k'$.
Without lost of generality, suppose $T$ is monochromatic in colour~$1$.

Note that we can label $E_M=(\{x_i,y_i\}\colon i=1\ldots\ell)$ such that $x_i$ and $x_j$ are at even distance in $T$ for $1\leq i<j\leq \ell$.
In fact, consider a proper colouring $\chi_T\colon V_T \to [2]$ and label the endpoints of edges in $E_M$ that are in
$\chi_T^{-1}(1)$ with $x_i$ and label the other endpoints with $y_i$.
Since any $x_i$ and $x_j$ are in the same colour class, they are at even distance in $T$.
We give labels $z_1,\ldots,z_{\ell'}$ to the vertices of $T$ that are not covered by $E_M$.

From the colouring of $E_R$ we know that, if $\chi_R(ij)=s$, then $|E_{G_s}(V_i,V_j)|\geq |V_i||V_j|/3$.
Therefore, 
since all the edges of $T$ are present in $R$, the pairs $\{V_i,V_j\}$ (for all $ij\in E_T$) are $\varepsilon$-regular in $G_1$ with $|E_{G_1}(V_i,V_j)|\geq |V_i||V_j|/3$. 

The graph composed of the classes $V_i$ for every $i\in V_T$ with edge set $E_{G_1}(V_i,V_j)$ between every pair has an $(\eps,1/3)$-regular $(\ell,\ell',k)$-cm-shape under $\chi_{K_{N,N}}$. 
\end{proof}

\section{Proof of the main result}\label{sec:main}

We start this section explaining the main ideas of the proof of Theorem~\ref{thm:threecolour}.
Given $\gamma>0$ and a sufficiently large $n$, we consider an arbitrary edge colouring of $K_{N,N}$ with $3$ colours, where $N={(3/2+\gamma)n}$.
Given a balanced graph $H=(V_H,E_H)$ on $n$ vertices, our aim is to show that there is a monochromatic copy of $H$ in $K_{N,N}$.
For this, we will apply the Embedding Lemma (Lemma~\ref{lemma:GeneralEmbedding}) to find a copy of $H$ in a suitable monochromatic subgraph $G$ of $K_{N,N}$.
To obtain this graph $G$ we first use Lemma~\ref{lemma:FindTheNiceTree} to find a monochromatic subgraph $G_T$ of $K_{N,N}$ that admits a partition in dense regular pairs, then we use the slicing results (Facts~ \ref{fact:Slicing} and~ \ref{fact:SuperSlicing}) to delete some vertices of $G_T$, obtaining $G$, which contains a regular partition with sufficiently dense and large regular pairs (that covers $\big(1+o(1)\big)n$ vertices) in a structured way.

To apply the Embedding Lemma successfully we need to prepare the graph $H$, showing that it admits a partition $\mathcal{W}$ compatible with a partition of $G$ in the sense of Definition~\ref{def:compatiblePartitions}.
Since $H$ has small bandwidth and is globally balanced, we use Lemma~\ref{lemma:balancedVectors} to order the classes of $\mathcal{W}$ to obtain some local balancedness.
This allow us to show that $\mathcal{W}$ is compatible with the regular partition of $G$ that we obtained before.
The proof is finished with the application of the Embedding Lemma.
In the next subsection we make this reasoning precise.
We remark that these ideas are similar to the ones used in the proof of~\cite[Theorem~1.3]{MoSaScTa15}.

\subsubsection*{Proof of Theorem \ref{thm:threecolour}}
Fix $\gamma>0$ and $\Delta\geq 1$, and let $\varepsilon_1$ be given by Lemma~\ref{lemma:GeneralEmbedding} applied with $d=1/4$ and $\Delta$.
Now set 
$$
\varepsilon=\min\left\{\frac{\varepsilon_1}{2}, \frac{\gamma/2}{24\cdot 10^5 + 2(3+\gamma/2)}\right\}.
$$
Since $\varepsilon\leq 1/(24\cdot 10^5)$, one can get $K_0$ from Lemma~\ref{lemma:FindTheNiceTree} applied with $\eps$.
Fix $\xi=\gamma/6$ and let $n_0$ be obtained from Lemma~\ref{lemma:balancedVectors} applied with $\xi$ and $K_0$. Set
\begin{equation*}
\beta=\frac{\varepsilon\xi(1+2\xi)}{72\Delta^2K_0^2}.
\end{equation*}
Let $H=(V_H,E_H)$ be an $n$-vertex balanced $(\beta,\Delta)$-graph and put 
\begin{equation*}
N=(3/2+\gamma)n,
\end{equation*}
where $n\geq \max\{n_0,8 K_0/\xi\}$. 
We assume that $n$ is divisible by $K_0$ (note that since $K_0$ is constant, this is not a problem).
To finish this starting preparation, consider an arbitrary $3$-colouring $\chi_{K_{N,N}}\colon E(K_{N,N})\to[3]$ of the edges of $K_{N,N}$.
In what follows we shall prove that the colouring $\chi_{K_{N,N}}$ yields a monochromatic copy of $H$.
For clarity, we split our proof into three parts: i) obtaining a suitable partition of a monochromatic subgraph $G$ of $V(K_{N,N})$; ii) obtaining a suitable partition of $V_H$; iii) application of the Embedding Lemma.

\ \\
\noindent\textit{Regular partition of a monochromatic subgraph $G$ of $V(K_{N,N})$}.
\ \\

We will obtain a well structured sufficiently regular monochromatic subgraph $G$ of $K_{N,N}$.
By Lemma~\ref{lemma:FindTheNiceTree}, we know that there exist integers $\ell,\ell',k$ with ${\ell,\ell'\leq k \leq K_0}$ and $\ell\geq k/(3+24\cdot 10^5\eps)$ such that $K_{N,N}$ has an $(\eps,1/3)$-regular $(\ell,\ell',2k)$-cm-shape under $\chi_{K_{N,N}}$.
This fact means that there exists a tree $T$ with vertices
$$
\{x_1,\dots,x_\ell,y_1,\dots,y_\ell,z_1,\dots,z_{\ell'}\}
$$
containing a matching $E_M=\{ x_i y_i\colon i=1,\ldots,\ell\}$ such that, for $1\leq i<j\leq \ell$, the vertices $x_i$ and $x_j$ are at an even distance in $T$, and there exists a partition $\{V_0,V_1,\ldots,V_k\}$ of ${V(K_{N,N})}$ such that the spanning subgraph of $K_{N,N}$ spanned by one of the colours (say colour 1), is $(\varepsilon,1/3)$-regular on $T$ and $|V_1|=\ldots=|V_k|=m$, where $m\geq (1-\varepsilon)N/k$.

Recall that $G_1$ is the the spanning subgraph of $K_{N,N}$ containing only edges with colour~$1$.
Let $G_T$ be the subgraph of $G_1$ induced by the classes in $\{V_1,\ldots,V_k\}$ that correspond to the vertices of $T$.
Applying Fact~\ref{fact:SuperSlicing} and Fact~\ref{fact:Slicing} in this order one can easily show that $G_T$ contains a subgraph 
\begin{equation*}
	\text{$G=(V_G,E_G)$ with a partition $\{A_1,\dots,A_\ell,B_1,\dots,B_\ell,C_1,\dots,C_{\ell'}\}$ of $V_G$}
\end{equation*}
such that each one of the classes in this partition have size at least $(1-\varepsilon)m$.
Furthermore, $A_1,\dots,A_\ell,B_1,\dots,B_\ell,C_1,\dots,C_{\ell'}$ correspond, respectively, 
to the vertices $x_1,\dots,x_\ell$, $y_1,\dots,y_\ell$, $z_1,\dots,z_{\ell'}$ of $T$.
Moreover, the bipartite graphs induced by $A_i$ and $B_i$ are $(2\varepsilon, 1/3-\varepsilon)$-super-regular.
Also, the bipartite graphs induced by the other pairs are $(2\varepsilon, 1/3-\varepsilon)$-regular.

Recall that ${N=(3/2+\gamma)n}$. Then, from the choice of $\eps$, since $m\geq(1-\varepsilon)N/k$ and $\ell\geq k/(3+24\cdot 10^5\eps)$, we conclude that
\begin{align}\label{eq:limitD}
|A_1|,\dots,|A_\ell|,|B_1|,\dots,|B_\ell|,|C_1|,\dots,|C_{\ell'}|	&\geq	(1-\eps)m\nonumber\\
		&\geq 	\frac{(1-\eps)^2 (3 + 2\gamma)}{3+24\cdot 10^5\eps}\left(\frac{n}{2\ell}\right)\nonumber\\
		&\geq 	\frac{3 + 2\gamma}{3+\gamma/2}\left(\frac{n}{2\ell}\right)\nonumber\\
		&\geq	(1+\gamma/3)\left(\frac{n}{2\ell}\right).
\end{align}

\ \\
\noindent\textit{Suitable partition of $V_H$.}
\ \\

In order to apply the Embedding Lemma we shall obtain a partition of $V_H$ that is compatible with the partition $\{A_1,\dots,A_\ell,B_1,\dots,B_\ell,C_1,\dots,C_{\ell'}\}$ of $V_G$.
This part of our proof is very similar to the corresponding one in~\cite{MoSaScTa15}.

Since $H$ is a balanced $(\beta,\Delta)$-graph, there exists a $2$-colouring ${\chi_H\colon V_H\to[2]}$ such that ${\big||\chi^{-1}(1)|-|\chi^{-1}(2)|\big|\leq \beta|\chi^{-1}(2)|}$ and there is a labelling $w_1,\ldots,w_n$ of $V_H$ such that $|i-j|\leq \beta n$ for every $w_i w_j\in E_H$.

Let $\hat\ell$ 
be the smallest integer dividing $n$ with ${\hat\ell\geq (7 K_0/\xi) + \ell\geq \ell(7/\xi + 1)}$.
Since $n$ is divisible by $K_0$, we know that
\begin{equation}\label{eq:Kohatell}
\hat\ell\leq (7 K_0/\xi) + 2K_0.
\end{equation}
Consider the partition $\mathcal{W}=\{W'_1,\ldots,W'_{\hat\ell}\}$ of $V_H$ with ${|W'_1|=\ldots= |W'_{\hat\ell}|=n/\hat\ell}$ taking the ordering $(w_1,\ldots,w_n)$ into account, i.e., 
$W'_i=w_{(i-1)n/\hat\ell + 1},\ldots,w_{in/\hat\ell}$
for ${i=1,\ldots,\hat\ell}$.
Recall that, given a permutation $\sigma\colon[\hat\ell]\to[\hat\ell]$, we denote by $C_i(\mathcal{W},\sigma,a,b)$ the number of vertices in $W'_{\sigma(a)}\cup W'_{\sigma(a+1)}\cup\dots\cup W'_{\sigma(b)}$ with colour $i$ under $\chi_H$.
Since $\beta\leq 2/\hat\ell$, one can use Lemma \ref{lemma:balancedVectors} to conclude that, for all integers $1\leq a<b\leq \hat\ell$ with $b-a\geq 7/\xi$, there exists a permutation 
$\sigma\colon[\hat\ell]\to[\hat\ell]$ such that
\begin{equation*}
|C_1(\mathcal{W},\sigma,a,b)-{C_2(\mathcal{W},\sigma,a,b)}|\leq \xi {C_2(\mathcal{W},\sigma,a,b)}.
\end{equation*}

Now we will define some families of the classes $W'_{\sigma(1)}, W'_{\sigma(2)},\ldots, W'_{\sigma(\hat\ell)}$.
For $1\leq i\leq \ell$, put ${a_i=(i-1)\hat\ell/\ell + 1}$ and ${b_i=i\hat\ell/\ell}$. 
Consider the families $\mathcal{W}_1,\ldots,\mathcal{W}_\ell$ such that $\mathcal{W}_i=\{W'_{\sigma(a_i)},W'_{\sigma(a_i+1)}, \ldots, W'_{\sigma(b_i)}\}$ for ${i=1,\ldots,\ell}$. We write $C_1(\mathcal{W}_i)$ for 
$C_1(\mathcal{W},\sigma,a_i,b_i)$ and $C_2(\mathcal{W}_i)$ for $C_2(\mathcal{W},\sigma,a_i,b_i)$.
Thus, for ${i=1,\ldots,\ell}$, since $b_i - a_i = \hat\ell/\ell + 1\geq 7/\xi$, we have
\begin{equation}\label{eq:equalcolours}
|C_1(\mathcal{W}_i)-{C_2(\mathcal{W}_i)}|\leq \xi {C_2(\mathcal{W}_i)},
\end{equation}

Recall that $\{x_1,\dots,x_\ell,y_1,\dots,y_\ell,z_1,\dots,z_{\ell'}\}$ are the vertices of the tree $T$.
Furthermore, $T$ contains the matching $E_M =\{ x_i y_i\colon i=1,\ldots,\ell\}$ and the distance between $x_i$ and $x_j$ in $T$ is even for all $i$ and $j$.
The partition of $V_H$ we will obtain is composed of classes $X_1,\dots,X_{\ell}$, $Y_1,\dots,Y_{\ell}$, $Z_1,\dots,Z_{\ell'}$ that correspond, respectively, to the vertices $x_1,\dots,x_\ell$, $y_1,\dots,y_\ell$, $z_1,\dots,z_{\ell'}$, which also correspond, respectively, to the classes $A_1,\ldots,A_\ell,B_1,\ldots,B_\ell$, $C_1,\ldots,C_{\ell'}$.

For every $i=1,\ldots,\ell$, most of the vertices of $\mathcal{W}_i$ will be part of the classes $X_i$ and~$Y_i$, depending on the colour given by $\chi_H$.
We will distribute the leftover vertices to allow a connection between the classes corresponding to matchings edges.

Each one of the classes $W'_i$ will be divided in two parts: the \emph{link} and the \emph{kernel}.
The link of $W'_i$, denoted by $L_i$, is the part that guarantees the connection between $W'_i$ and $W'_{i+1}$ if it is necessary.
For the class $W'_{\hat\ell}$, we set $L_{\hat\ell}=\emptyset$, and 
for $1\leq i\leq \hat\ell-1$, if $W'_i$ and $W'_{i+1}$ are in the same class $\mathcal{W}_r$, then 
$L_i=\emptyset$.
In what follows we define the link $L_i$ of $W'_i$ for $1\leq i\leq \hat\ell-1$ when $W'_i\in \mathcal{W}_r$ and $W'_{i+1}\in \mathcal{W}_{s}$ for $r\neq s$.
Denote by $P_T(r,s)$ the path of the tree $T$ between $x_r$ and $x_s$ and let $P^{\int}_T(r,s)\subset P_T(r,s)$ be the path obtained by excluding the vertices in $\{x_r,y_r,x_s,y_s\}$ from $P_T(r,s)$.
For simplicity, set $t_{r,s}=|P^{\int}_T(r,s)|$.
We divide the ${(t_{r,s}+1)\beta n}$ ``last'' vertices of $W'_i$ in $t_{r,s}+1$ ``pieces'' of size $\beta n$, where the $j$-th piece is denoted by $L_i(j)$ for ${1\leq j\leq t_{r,s}+1}$, that is,
\begin{equation*}
L_i(j)=w_{(i-(t_{r,s}+2-j)\beta \hat\ell)n/\hat\ell+1},\ldots,w_{(i-(t_{r,s}+1-j)\beta \hat\ell)n/\hat\ell}.
\end{equation*}
Finally, we define the link $L_i$ of $W'_i$ as $L_i = \{L_i(1),\ldots,L_i(t_{r,s}),L_i(t_{r,s}+1)\}$.

The kernel of $W'_i$, denoted by $\KE_i$, is the set of the remaining vertices of $W'_i$, i.e., $\KE_i = W'_i \setminus 
L_i$.
All vertices of $\KE_i$ will belong to the matching classes $X_i$ and $Y_i$.

We are now ready to form the partition of $V_H$ that we need.
We start with empty classes $X_1,\dots,X_{\ell},Y_1,\dots,Y_{\ell},Z_1,\dots,Z_{\ell'}$.
Let us first deal with the kernels.
For every $1\leq i\leq \ell$, take each class $W'_p$ of the family $\mathcal{W}_i$, put in $X_i$ all the vertices $w$ of 
the kernel $\KE_p$ with $\chi_H(w)=1$ and put in $Y_i$ all the vertices $w$ of $\KE_p$ with $\chi_H(w)=2$.

Now we will distribute the vertices of the links.
For every ${1\leq i\leq \hat\ell-1}$ such that $W'_i\in \mathcal{W}_r$ and $W'_{i+1}\in \mathcal{W}_{s}$ with $r\neq s$, let $\{u_1,\ldots,u_{t_{r,s}}\}$ be the vertices of $P^{\int}_T(r,s)$.
Furthermore, let $u_0$ and $u_{t_{r,s}+1}$ be, respectively, the vertices of $T$ adjacent to $u_1$ and $u_{t_{r,s}}$ in $P_T(r,s)$.
We will show that is possible to ``walk'' between the matching classes.
The vertex $u_0$ can be either $x_r$ or $y_r$.
Without loss of generality we assume that $u_0=x_r$.
For $1\leq j\leq t_{r,s}+1$, we put the vertices $w$ of $L_i(j)$ with $\chi_H(w)=1$ in the class corresponding to $u_{j-1}$ if $j$ is even, and in the class corresponding to $u_j$ if $j$ is odd.
On the other hand, we put the vertices $w$ with $\chi_H(w)=2$ in the class corresponding to $u_{j}$ if $j$ is even, and in the class corresponding to $u_{j-1}$ if $j$ is odd.
Recall that $x_i$ and $x_j$ are at an even distance in $T$ for all $1\leq i<j\leq \ell$.
Furthermore, every link has size $\beta n$.
Therefore, we know that there is no edges inside the classes and if there is an edge between two
classes, then the corresponding edge belongs to the tree $T$.

\ \\
\noindent\textit{Application of the Embedding Lemma}.
\ \\

In this final part we will apply the Embedding Lemma (Lemma~\ref{lemma:GeneralEmbedding}) to find a copy of $H$ in $G$, which is a monochromatic subgraph of $K_{N,N}$.
For this, we will prove that the partition $\{X_1,\dots,X_{\ell},Y_1,\dots,Y_{\ell},Z_1,\dots,Z_{\ell'}\}$  of $V_H$ is 
$(2\varepsilon_1,T,M)$-compatible with the partition  $\{A_1,\dots,A_\ell,B_1,\dots,B_\ell,C_1,\dots,C_{\ell'}\}$ of $V_G$. 

We start bounding from above the size of each class in
$\{X_1,\dots,X_{\ell},Y_1,\dots,Y_{\ell},Z_1,\dots,Z_{\ell'}\}$.
Since ${C_1(\mathcal{W}_i)+C_2(\mathcal{W}_i) = n/\ell}$ for every $1\leq i\leq \ell$, we can use~\eqref{eq:equalcolours} to notice that, for every $1\leq i\leq \ell$, we have
\begin{equation}
(1-\xi)\frac{n}{2\ell} \leq C_1({\mathcal{W}_i}), C_2({\mathcal{W}_i}) \leq (1+\xi)\frac{n}{2\ell}.
\end{equation}
The choice of $\beta$ and $\eps$ combined with~\eqref{eq:Kohatell} implies that
\begin{equation}\label{eq:4llbeta}
4\ell\hat\ell\beta \leq \frac{(1+\gamma/3)\eps}{2\Delta^2}\leq \xi.
\end{equation}

Let $S_{\min}$ be the set in $\{A_1,\dots,A_\ell,B_1,\dots,B_\ell,C_1,\dots,C_{\ell'}\}$ with minimum cardinality.
Note that, for $1\leq i\leq \ell$, the classes $X_i$ and $Y_i$ are composed, respectively, of vertices $v$ with $\chi(v)=1$  and $\chi(v)=2$.
Since these vertices can come from one kernel and at most two pieces of each link, we can use~\eqref{eq:limitD} and~\eqref{eq:4llbeta} to obtain
\begin{equation}\label{eq:limitXY}
|X_i|,|Y_i| \leq (1+\xi)\frac{n}{2\ell} + 2\hat\ell \beta n
=\left(1+\xi+  4\ell\hat\ell\beta\right) \frac{n}{2\ell}
=(1+2\xi) \frac{n}{2\ell}
\leq|S_{\min}|.
\end{equation}
To bound the size of the classes $Z_i$ from above, for $1\leq i\leq \ell'$, note that they are composed only of vertices in at most two pieces of each link.
Then, from the choice of $\beta$ and $\hat\ell$, and using~\eqref{eq:limitD} and~\eqref{eq:4llbeta}, we have
\begin{equation}\label{eq:limitZ}
|Z_i| \leq 2\hat\ell\beta n
=(4\ell\hat\ell\beta)\frac{n}{2\ell}
\leq \frac{\varepsilon}{2\Delta^2}|S_{\min}|.
\end{equation}

We are ready to check that $\{X_1,\dots,X_{\ell},Y_1,\dots,Y_{\ell},Z_1,\dots,Z_{\ell'}\}$ and $\{A_1,\dots,A_\ell,B_1,\dots,B_\ell$, $C_1,\dots,C_{\ell'}\}$ are $(2\varepsilon_1,T,M)$-compatible.
Following the setup of Definition~\ref{def:compatiblePartitions}, define sets $U_i$ and $U'_i$ for $1\leq i\leq 2\ell+\ell'$ with respect to the partition $\{X_1,\dots,X_{\ell},Y_1,\dots,Y_{\ell},Z_1,\dots,Z_{\ell'}\}$, and let $W_i$ be as follows, for $1\leq i\leq 2\ell+\ell'$:
\begin{equation*}
		W_i=
		\begin{cases}
			X_i, 		&\text{if $1\leq i\leq \ell$},\\
			Y_{i-\ell} 	&\text{if $\ell+1\leq i\leq 2\ell$},\\
			Z_{i-2\ell}	&\text{if $2\ell+1\leq i\leq 2\ell+\ell'$}.
		\end{cases}
\end{equation*}
It is left to verify that the conditions of Definition \ref{def:compatiblePartitions} hold:

\begin{itemize}
 \item [(i)] From the construction of the partition $\{X_1,\dots,X_{\ell},Y_1,\dots,Y_{\ell},Z_1,\dots,Z_{\ell'}\}$ of $V_H$, if there is an edge between two classes, then the corresponding edge belongs to the tree~$T$.

\item [(ii)] The validity of this condition follows directly from~\eqref{eq:limitXY} and~\eqref{eq:limitZ}.

\item [(iii)] Fix $1\leq i\leq 2\ell+\ell'$. 
By definition, the set $U_i$ is composed of the vertices of $W_i$ with neighbours in some $W_j$ 
with $i\neq j$ such that $ij$ is not an edge of the matching $E_M$.
We consider two cases depending on the value of $i$:
\begin{itemize}
\item [(a)]$2\ell+1\leq i\leq 2\ell+\ell'$: Here, $U_i=Z_{i-2\ell}$. From~\eqref{eq:limitZ}, we have $|U_i|\leq \varepsilon 
|S_{\min}|/\Delta$.
\item [(b)]$1\leq i\leq 2\ell$: In this case, $U_i$ is either $X_i$ (if $1\leq i\leq \ell$) or $Y_{i-\ell}$ (if $\ell+1\leq i\leq 2\ell$).
From the construction of the partition $\{X_1,\dots,X_{\ell},Y_1,\dots,Y_{\ell},Z_1,\dots,Z_{\ell'}\}$, we know that the classes $X_i$ and $Y_{i-\ell}$ have neighbours in no more than two of the classes in $\{Z_1,\ldots,Z_{\ell'}\}$.
Then, since $\Delta$ is the maximum degree of $H$, we conclude from~\eqref{eq:limitZ} that 
$$
|U_i|\leq 2\Delta  \left(\frac{\varepsilon |S_{\min}|}{2\Delta^2}\right) = \frac{\varepsilon}{\Delta} |S_{\min}|.
$$
\end{itemize}
Therefore, for every $i=1,\ldots,2\ell+\ell'$ we have
\begin{equation}\label{eq:UjDmin}
|U_i|\leq \frac{\varepsilon}{\Delta} |S_{\min}|,
\end{equation}
which verifies condition (iii). 

\item [(iv)] Recall from Definition~\ref{def:compatiblePartitions} that $U=\bigcup_{i=1}^{2\ell+\ell'}U_i$ and $U'_i=N_H(U)\cap (W_i\setminus U)$, i.e., $U_i'$ is the subset of $W_i$ composed of the neighbours of vertices in $U$ that are not in $U$.
Again, we consider some cases depending on the value of $i$.
\begin{itemize}
\item[(a)] $1\leq i\leq \ell$: In this case, $U'_i\subset W_i= X_{i}$.
From the definition of $U$, vertices of $X_{i}$ that have neighbours in $Z_1\cup\ldots\cup Z_{\ell'}$ belong to $U$.
Since $U'_i\subset X_{i}$ and $U'_i$ contains no vertex from $U$, we conclude that $U'_i$ is composed only of neighbours of $U_{i}\subset Y_{i}$.
Therefore, from~\eqref{eq:UjDmin}, we have ${|U'_i|\leq \Delta|U_{i}|\leq \varepsilon |S_{\min}|}$.
\item [(b)]$\ell+1\leq i\leq 2\ell$: In this case, $U'_i\subset W_i= Y_{i-\ell}$.
From the definition of $U$, vertices of $Y_{i-\ell}$ that have neighbours in $Z_1\cup\ldots\cup Z_{\ell'}$ belong to $U$.
Since $U'_i\subset Y_{i-\ell}$ and $U'_i$ contains no vertex from $U$, we conclude that $U'_i$ is composed only of neighbours of $U_{i-\ell}\subset X_{i-\ell}$.
Therefore, from~\eqref{eq:UjDmin}, we have ${|U'_i|\leq \Delta|U_{i-\ell}|\leq \varepsilon |S_{\min}|}$.
\item [(c)]$2\ell+1\leq i\leq 2\ell+\ell'$: In this case $U_i=Z_{i-2\ell}$.
Then, $U'_i=\emptyset$.
Therefore, clearly we have $|U'_i|\leq \eps|S_{\min}|$.

\end{itemize}
\end{itemize}

We just proved that the partition $\{X_1,\dots,X_{\ell},Y_1,\dots,Y_{\ell},Z_1,\dots,Z_{\ell'}\}$ of
$V_H$ is $(2\varepsilon,T,M)$-compatible with the partition
$\{A_1,\dots,A_\ell,B_1,\dots,B_\ell,C_1,\dots,C_{\ell'}\}$ of $V_G$.
Since $\eps\leq \eps_1/2$, this fact implies $(\varepsilon_1,T,M)$-compatibility.
Therefore, Lemma~\ref{lemma:GeneralEmbedding} guarantees that $H\subset G$, which concludes the proof.
\qed

\bibliography{minhabib}

\end{document}